\documentclass[12pt]{amsart}
\usepackage{amsfonts, mathdots}
\usepackage{geometry}                
\usepackage{graphicx}
\usepackage{amssymb}
\usepackage{epstopdf}
\usepackage{verbatim, color}
\usepackage{mathrsfs}
\allowdisplaybreaks
\renewcommand{\baselinestretch}{1.2}
\renewcommand{\arraystretch}{1.2}

\newcommand{\bfa}{{\mathbf{a}}}
\newcommand{\bfg}{{\mathbf{g}}}
\newcommand{\bfh}{{\mathbf{h}}}
\newcommand{\bfb}{{\mathbf{b}}}
\newcommand{\bfx}{{\mathbf{x}}}
\newcommand{\bfy}{{\mathbf{y}}}
\newcommand{\bfz}{{\mathbf{z}}}

\newcommand{\wtk}{{\widetilde{k}}}
\newcommand{\wtn}{{\widetilde{n}}}

\newcommand{\rightarrowp}{{\succ^+}}
\newcommand{\Idemp}{{\sf Idemp}}
\newcommand{\Elim}{{\sf Elim}}
\newcommand{\QI}{{\sf QI}}
\newcommand{\EN}{{\sf EN}}
\newcommand{\EZ}{{\sf EZ}}
\newcommand{\NZ}{{\sf NZ}}
\newcommand{\ZN}{{\sf ZN}}
\newcommand{\NENZ}{{\sf NENZ}}
\newcommand{\FALSE}{{\sf FALSE}}
\newcommand{\bbZ}{{\mathbb Z}}
\newcommand{\fM}{{\frak M}}

\newcommand{\PS}{{\mathcal{PS}}}
\newcommand{\Valid}{{\sf Valid}}
\newcommand{\SAT}{{\sf SAT}}

\newcommand{\True}{{\sf TRUE}}
\newcommand{\False}{{\sf FALSE}}

\newcommand{\IT}[1]{{#1}^\star}
\newcommand{\ITI}{{I^\star}}

\newcommand{\SM}{{\mathcal{SM}}}
\newcommand{\bSM}{{\mathbf{SM}}}

\newcommand{\CE}{{\mathcal{CE}}}
\newcommand{\bbI}{{\mathbb{I}}}

\newcommand{\bPS}{{\mathbf{PS}}}
\newcommand{\PSA}{{\mathcal{PSA}}}
\newcommand{\SMA}{{\mathcal{SMA}}}
\newcommand{\pmbbZ}{\pmb{\mathbb Z}}

\newcommand{\plus}{{\;\cup\;}}
\newcommand{\summe}{\bigcup}
\newcommand\prodkt{\bigcap}

\newtheorem{theorem}{Theorem}[section]

\newtheorem{lemma}[theorem]{Lemma}
\newtheorem{corollary}[theorem]{Corollary}

\newtheorem{definition}[theorem]{Definition}
\newtheorem{example}[theorem]{Example}

\newtheorem{remark}[theorem]{Remark}

 \newcounter{thlistctr}
 \newenvironment{thlist}{\
 \begin{list}%
 {\alph{thlistctr}}%
 {\setlength{\labelwidth}{2ex}%
 \setlength{\labelsep}{1ex}%
 \setlength{\leftmargin}{6ex}%
 \usecounter{thlistctr}}}%
 {\end{list}}

\newsymbol \ndiv   232D
\newsymbol \nmodels  2332

\newcommand{\bfv}{{\mathbf v}}

\newcommand{\bA}{{\mathbf A}}
\newcommand{\bB}{{\mathbf B}}

\newcommand{\cC}{{\mathcal C}}

\newcount\minute        
\newcount\hour          
\newcount\hourMins  
\def\now%
{
  \minute=\time    
  \hour=\time \divide \hour by 60 
  \hourMins=\hour \multiply\hourMins by 60
  \advance\minute by -\hourMins 
  \zeroPadTwo{\the\hour}:\zeroPadTwo{\the\minute}%
}
\def\timestamp%
{
  \today\ \now
}
\def\zeroPadTwo#1%
{
  \ifnum #1<10 0\fi    
  #1
}

\title[Boole's Method I.]{Boole's Method I.\\ A Modern Version}
\author{Stanley Burris}
\author{H.P. Sankappanavar}
\date{\today}                                           

\begin{document}
\maketitle

\begin{abstract}
A rigorous, modern version of Boole's algebra of logic is presented, 
based partly on the 1890s treatment of Ernst Schr\"oder. 

\end{abstract}

\section{{Preamble to Papers I and II}}

The sophistication and mathematical depth of Boole's approach to the logic of 
classes is not commonly known, not even among logicians. 
It includes much, much more than just the basic operations and equational laws 
for an algebra of classes. Indeed, aside from possibly a few tricks to speed up
computations, Boole considered his algebra of logic to be the 
perfect completion of the fragmentary Aristotelian logic. 
Whereas the latter consisted of a small catalog of valid arguments, Boole's system 
offered a method (consisting of algebraic algorithms) to determine 
\begin{enumerate}
\item[(B1)] 
the strongest possible conclusion 
$\varphi(\vec{A}\:)$ from {\it any\/} given finite 
collection of premisses $\varphi_i(\vec{A},\vec{B}\:)$ concerning classes 
$A_1, \ldots, A_m$, $B_1,\ldots, B_n$, and
\item[(B2)] 
the expression of any class $A_i$ in terms of the other classes in {\it any\/} given finite 
collection of premisses $\varphi_i(\vec{A}\:)$ concerning classes $A_1, \ldots, A_m$.
\end{enumerate}
Boole's algebra of logic was developed well before concerns were raised about 
possible paradoxes in the study of classes.
To maintain contact with Boole's writings, as well as with modern set-theoretic 
foundations and notations, 
we will simply treat the words `class' and `set' as equivalent. 
To put everything into a more modern form,
simply change the word `class' everywhere into the word `set'.

With Boole's algebraic approach, the mastery of the logic of classes changed dramatically
from the requirement of memorizing a finite and
 very incomplete catalog of valid arguments in Aristotelian logic to the requirement of learning: 
 \begin{enumerate} 
 \item[(a)]
how to translate class-propositions into class-equations, and vice-versa, 
\item[(b)]
the axioms and rules of inference for Boole's algebra of logic, and  
\item[(c)]
the fundamental theorems of Boole's algebra of logic.
\end{enumerate}
 
Boole never precisely stated which ordinary language statements 
$\varphi_i(\vec{A}\:)$ qualified as class-propositions, that is, 
propositions about classes, although he gave many examples.
By 1890 Schr\"oder concluded that any class-proposition was equivalent 
to a {\em basic formula} in the modern Boolean algebra of sets, that is,
 either to an equational assertion 
$p(\vec{A}\:) = q(\vec{A}\:)$, or the negation
$p(\vec{A}\:) \neq q(\vec{A}\:)$
of an equational assertion.  The four forms of 
 categorical propositions from Aristotelian logic are readily 
 seen to satisfy this condition:
 $$
 \begin{array}{c | l | c}
\text{Form} & \text{Ordinary language} & \text{Equational form}\\
 \hline
\text{A} & \text{All } A \text{ is } B & A = A\cap B\\
\text{E} & \text{No } A \text{ is } B & A\cap B = \O\\
\text{I} & \text{Some } A \text{ is } B & A \cap B \neq \O\\
\text{O} & \text{Some } A \text{ is not } B & A\cap B' \neq \O
\end{array}
$$
 
 Boole allowed more complex assertions, such as 
`All $A$ is $B$ or $C$', which can be expressed by $A = A\cap (B\cup C)$.
The converse, that every basic formula $\beta(A_1,\ldots,A_m)$
 in the modern Boolean algebra of sets can be expressed by a proposition in ordinary language, 
 is not so clear---ordinary language 
suffers from not using parentheses to group terms.
For example, it is cumbersome to express $A \cap (B\cup (C \cap D))) = \O$ in ordinary
language; but with parentheses it is easy, namely 
`The class $A$ intersected with the class 
($B$ unioned with the class ($C$ intersected with the class $D$)) is empty'.
Without parentheses one needs the cumbersome method of introducing new symbols,
for example, `There are classes $E$ and $F$ such that $F$ is the intersection of $C$ and $D$,
and $E$ is the union of $B$ and $F$, and $A$ and $E$ are disjoint'. 

{\it We will simply assume that class-propositions correspond precisely to basic formulas in the modern Boolean algebra of classes. Furthermore we assume that
the reader knows how to translate between class-propositions and basic formulas.}

The word `algebra' has two major meanings in mathematics---we first learn to think of algebra as 
{\em procedures}, such as finding the roots of a quadratic equation; later we learn that it can also refer 
to a {\em structure} such as the ring of integers $\pmbbZ = (\bbZ, +,\cdot, -, 0, 1)$, or the power set
algebra 
$\bPS(U) = (\PS(U),\cup, \cap,',\O,U)$ of subsets of $U$.

When Boole introduced and refined his algebra of logic for classes, from 1847 to 1854, 
he was primarily interested in procedures to determine the items in (B1) and (B2) above. 
Given a finite list of class-propositions 
$\varphi_i$
for the premisses, the first step was to convert them into {\em equations}
$p_i= q_i$.
(Note: Schr\"oder thought it was necessary to use basic formulas, not just equations. 
Boole believed he only needed equations.)
Then he gave algebraic algorithms for (B1) and (B2) in the setting of equations.
 The result was then
translated back into ordinary language to give the desired class-proposition conclusion.
%

\begin{remark}
The reader can find a detailed presentation of Boole's algorithms with examples, but without proofs, 
in the article {\em George Boole}, 
in the online Stanford Encyclopedia of Philosophy \cite{Burris-SEP}.
\end{remark}

Boole's version of the algebra of logic for classes was significantly different from
what we now call Boolean algebra---but 
it led directly to modern Boolean algebra, thanks to Jevons \cite{Jevons-1864} replacing Boole's partial operations by total operations.  
Scholars had from the very beginning at least three major concerns 
about Boole's system: 
\begin{enumerate}
\item
It seemed unduly and mysteriously tied to the algebra of 
numbers, the so-called {\it common algebra}---for the fundamental operations on classes, and the fundamental constants, 
Boole chose the symbols 
$+$, $\cdot$, $-$, 0 and 1, symbols traditionally reserved for the algebra of numbers. 
(Boole also used
division, but only in a very special setting.)
 His manipulation of equations was dictated by 
the procedures used
in common algebra, {\bf with one addition:} multiplication  was idempotent for {\em class-symbols}, 
that is, $A^2 = A$ for any class-symbol $A$.
\item
Boole interpreted 0 as the empty class, 1 as the universe, and the multiplication of classes as
their intersection. But his operations of addition ($+$) and subtraction ($-$) on classes 
were {\em partial operations}, not total operations; that is, they were only partially defined. 
If two classes $A$ and $B$ had elements in common, then $A+B$ was simply 
not defined
(or, as Boole said, $A+B$ was not interpretable). Likewise, if $B$ was not a subclass of $A$, 
then $A-B$ was not defined; otherwise $A-B$ was $A \cap B'$, the class of elements in $A$
but not in $B$.

The difficulty readers had with Boole's partial operations was that Boole applied the processes 
of common algebra to
equations without being concerned about whether the terms were defined or not. 
In modern universal algebra we know that the usual rules of equational inference
(Birkhoff's five rules) are {\em correct} and {\em complete} for the equational logic 
of {\em total} algebras, that is, algebras with fundamental operations that are totally 
defined on the domain of the algebra. Unfortunately these properties may not hold 
when working
with partial algebras. 

With Boole's system, the question was whether or not
the application of the usual rules of equational inference always leads to correct results
when one starts with meaningful premisses and ends with a meaningful conclusion, but
not all the equations appearing in the intermediate steps are meaningful. (Boole claimed
that the answer was `yes'.)
\item
Boole claimed that he could translate particular propositions into equations by introducing a 
new symbol $V$. 
For example,
`Some $A$ is $B$' was translated initially by $V = AB$, and later by $VA = VB$. 
\end{enumerate}
Items (1) and (2) remained troublesome issues for more than a century, until the appearance of 
Hailperin's book 
\cite{Hailperin-1976} in 1976. 
He set these concerns aside
by noting that each partial algebra $\bB(U) = (\PS(U),+,\cdot,-,0,1)$ in Boole's setting 
could be embedded in a total algebra of signed multi-sets;
this is equivalent to saying that $\bB(U)$ can be embedded in the ring $\pmbbZ^U$ (see, for example, \cite{BuSa-Bulletin}).
Regarding item (3),
Schr\"oder `proved' that one had to use negated equations for propositions with existential import.
 (This approach made the introduction of a new symbol $V$ quite unnecessary).
Item 3 has remained a concern...we will show that Boole's view, that only equations are needed,  
is actually correct as well (that is, after we make a very small adjustment to his translations 
between class-properties and  class-equations).

Boole's algorithms are powerful tools in the study of classes, and they 
carry over almost verbatim to the setting of modern Boolean algebra. 
We will adapt Boole's algebra of logic for classes (his theorems and algorithms)
 to the modern setting in this paper, essentially
along the lines laid out in 
the 1890s by Schr\"oder. This will allow the reader to understand
and judge the importance of Boole's work, without the hinderance of possibly many 
nagging concerns  regarding whether or not one has properly understood 
all the nuances of meaning in Boole's writings. 
This modern version of Boole's work will include a discussion of item (3) above, showing that
indeed one only needs equations (refuting Schr\"oder's claim to have proved the contrary).


 In the second paper we turn to Boole's original system (compactly presented in the
 aforementioned SEP article) 
 and provide full details of the proofs (using the results of this first paper), including addressing item (3) above. 
 The controversy-free presentation in this first paper will hopefully make it easier for the reader 
to focus in the second paper on how the concerns regarding (1) and (2)
 in Boole's system are overcome.  Furthermore this first paper sets the stage for how we will resolve the concerns about
  item (3) in Boole's system.

 In closing this Preamble, we would like to mention that, in \cite{BuSa-Bulletin}, Boole's claim that his ``Rule of $0$ and $1$'' is sufficient to prove his theorems is vindicated.                                      

\subsection{Introduction}

More specifically, Boole's algebra of logic (1847/1854) offered 
\begin{itemize}
\item
a translation of propositions into equations, 
\item
an algorithm for eliminating symbols in the equations,  
\item
an algorithm for solving for a variable, and
\item
a reverse translation, from conclusion equations 
to conclusion propositions.
\end{itemize} 
This algebra of logic has long puzzled readers for many reasons, including: 
\begin{thlist}
\item
its foundation, which appears to be the `common' algebra, namely the algebra 
of numbers, augmented by idempotent variables, 
\item
the appearance of uninterpretable terms in various procedures, 
\item
a strange division procedure, 
\item
a dubious encoding of propositions as equations, especially the particular 
propositions (using his famous $V$), and 
\item
dubious proofs of the main theorems. 
\end{thlist}
Yet the system seemed, by and large, to work just as Boole said it 
would.\footnote
{In 1864 Jevons \cite{Jevons-1864} modified Boole's system, giving the basic structure that 
would develop into modern Boolean algebra.
}
The mechanical details of Boole's method of using algebra to analyze arguments are given 
in considerable detail in the article ``George Boole'' in the online Stanford Encyclopedia of Philosophy 
(see \cite{Burris-SEP}). Now we turn to the justification of his method.

This first paper gives a compact yet rigorous modern version of Boole's 
algebra of logic. It is based in good part on Volumes I and II of Schr\"oder's 
{\em Algebra der Logik} \cite{Schr}, published in the 1890s. These results, along with 
the remarkable insights of Hailperin (\cite{Hailperin-1976} 1976/1986), are used in the second 
paper (\cite{BuSaII}) to likewise give a compact, rigorous presentation of  Boole's original 
algebra of logic.\footnote
{Brown \cite{Brown-2009} has given a fairly compact treatment of the development of the algebra 
used by Boole, showing that a certain ring (a ring of polynomials modulo idempotent generators) 
satisfies Boole's theorems. However this does not show that Boole's algebra of logic gives a correct 
calculus of classes, as Boole claimed, and the author seems to suggest. 
Boole's algebra of logic has partial operations, and one cannot simply apply Birkhoff's rules of 
equational logic to partial algebras. Hailperin \cite{Hailperin-1976} extended Boole's partial algebra 
to a total algebra of signed multisets, and for such an algebra Birkhoff's rules apply---this, or some 
step connecting Boole's partial algebra to Birkhoff's rules, is missing in Brown's treatment. Hailperin's 
work falls short of being complete 
by 
the absence of his justification to Boole's use of equations to express particular 
propositions--we address this in our second paper.
} 

\section{Background}
%

For purposes of indexing we prefer to use sets of the form 
$\widetilde{n} := \{1,\ldots,n\}$ instead of the usual finite ordinals
$n := \{0,\dots,n-1\}$.

Given a universe $U$, the {\bf power set} $PS(U)$ of $U$ is the set
of subsets of $U$. The {\bf power-set algebra} $\bPS(U)$ is the algebra
$(PS(U),\plus ,\cap , ', \O, U)$ of subsets of $U$. 
$\PSA$ is the collection of
power-set algebras $\bPS(U)$ with $U\neq \O$.

For the syntactic side of power-set algebra we use 
the {\bf operation symbols}  $\cup  $ ({\bf union}), 
$\cap$ ({\bf intersection}) and $'$ ({\bf complement}); 
and the {\bf constants}  0 (the {\bf empty set}) 
and 1 (the {\bf universe}).
There is a countably infinite set $X$ of {\bf variables}, and the 
$\PSA$-{\bf terms} $p(\bfx) := p(x_1,\ldots, x_k)$ are constructed from 
the above operation symbols, constant symbols and variables, in the 
usual way by induction:
\begin{itemize}
\item
 variables and constants are $\PSA$-terms;
\item
 if $p$ and $q$ are $\PSA$-terms then so are $(p')$, $(p\plus  q)$,  and $(p\cap q)$.
\end{itemize}
We adopt the usual convention of not writing outer parentheses.
We often write $p\cdot q$, or simply $pq$, instead of $p\cap q$. 
It will be assumed that intersection takes precedence over union, 
for example, $p \cup qr$ means $p \cup (q\cap r)$.
It will be convenient
to adopt the {\bf abbreviations} $p \subseteq q$ and $q \supseteq p$
for $p = pq$, or equivalently, $pq'=0$.

{\bf (First-order) $\PSA$-formulas} are defined inductively:
\begin{itemize}
\item
{\bf $\PSA$-equations}, that is, expressions of the form $(p=q)$,  where $p$ and $q$
are $\PSA$-terms, are $\PSA$-formulas
\item
if $\varphi$ is a $\PSA$-formula then so is $(\neg\,\varphi)$
\item
if $\varphi_1$ and $\varphi_2$ are $\PSA$-formulas, then so are 
$(\varphi_1 \wedge \varphi_2)$,  $(\varphi_1 \vee \varphi_2)$,
$(\varphi_1 \rightarrow \varphi_2)$, and $(\varphi_1 \leftrightarrow \varphi_2)$
\item
if $\varphi$ is a $\PSA$-formula and $x\in X$, then $\big((\forall x)\varphi\big)$ and 
$\big((\exists x)\varphi\big)$ are $\PSA$-formulas.
\end{itemize}
Again we adopt the usual convention of not writing outer parentheses. The notation
$(\exists \bfx)$ stands for $(\exists x_1)\cdots(\exists x_k)$, where $\bfx$ is the list
$x_1,\ldots,x_k$.

An {\bf interpretation $I$ into $\bPS(U)$} is a mapping $I : X \rightarrow PS(U)$
that is extended by induction to all terms as follows: 
\begin{itemize}
\item
$I(0) := \O$,\  $I(1) := U$
\item
 $I(p') := I(p)'$
\item
$I(p\plus  q) := I(p) \plus   I(q)$
\item
$I(p  q) := I(p)\cap I(q)$.
\end{itemize}
$I$ is a {\bf $\PSA$-interpretation} if it is an interpretation into some $\bPS(U)$.

The notion of a (first-order) $\PSA$-formula $\varphi$ being true under an interpretation $I$
into $\bPS(U)$,
written $I(\varphi) = \True$, 
is recursively defined as follows, where $I(\varphi)=\False$ means $I(\varphi)\neq \True$:
\begin{itemize}
\item
$I(p = q) = \True$ iff  $I(p) = I(q)$; 

\item
$I(\neg \varphi) = \True $ iff $I(\varphi) = \False$; 

\item
$I(\varphi \vee \psi) = \True$  iff either $I(\varphi) = \True$  or
$I(\psi) = \True$; 

\item
$I(\varphi \wedge \psi) = \True$ iff both $I(\varphi) = \True$ and
$I(\psi) = \True$; 

\item
$I(\varphi \rightarrow \psi) = \True$ iff  $I(\varphi) = \False$ or
$I(\psi) = \True$; 

\item
$I(\varphi \leftrightarrow \psi) = \True$ iff  both $I(\varphi\rightarrow \psi) = \True$ 
and $I(\psi \rightarrow \varphi) = \True$;

\item
$I\big( (\forall x) \varphi \big) = \True$ iff for each interpretation $\widehat{I}$ into $\bPS(U)$
 that agrees
with $I$ on $X\smallsetminus \{x\}$, one has $\widehat{I}(\varphi) = \True$;

\item
$I\big( (\exists x) \varphi \big) = \True$ iff for some interpretation $\widehat{I}$ into $\bPS(U)$
 that agrees
with $I$ on $X\smallsetminus \{x\}$, one has $\widehat{I}(\varphi) = \True$.
\end{itemize}
Note that $I(p\subseteq q) = \True$ iff $I(q\supseteq p) = \True$ iff $I(p) \subseteq I(q)$.

Some additional notation that we will use is:
\begin{itemize}
\item
$\PSA \models \varphi$, read {\bf $\PSA$ satisfies $\varphi$},
 means $I(\varphi)=\True$ for every interpretation $I$ into a member of $\PSA$.
 
 \item
$\varphi_1,\ldots,\varphi_n \models_\PSA \psi$, or 
$\varphi_1,\ldots,\varphi_n \Rightarrow_\PSA \psi$,
read  {\bf $\PSA$ (semantically) implies $\varphi$},
means
$\PSA \models \left(\varphi_1 \wedge \cdots \wedge \varphi_n  \rightarrow \psi\right)$.

\item
 $\varphi$ and $\psi$ are $\PSA$-{\bf (semantically) equivalent}, 
written $\varphi \Leftrightarrow_\PSA \psi$, if 
$\PSA \models \big(\varphi \leftrightarrow \psi\big)$.

\item
Two finite sets (or lists) $\Phi$ and $\Psi$ of formulas are $\PSA$-{\bf (semantically) equivalent}, 
written $\Phi \Leftrightarrow_\PSA \Psi$, if 
$\PSA \models \Big(\bigwedge\Phi \leftrightarrow \bigwedge\Psi\Big)$.           

\item
A finite set (or list) $\Phi$ of $\PSA$-formulas is {\bf $\PSA$-satisfiable}, written
$\SAT_\PSA\big(\Phi \big)$, if there is a $\PSA$-interpretation $I$
such that $I\Big(\bigwedge  \Phi \Big) = \True$.

\item
An argument  $\varphi_1,\ldots,\varphi_n\ \therefore\ \psi$
is {\bf $\PSA$-valid}, or {\bf valid in $\PSA$}, also written as 
$\Valid_\PSA \big(\varphi_1,\ldots,\varphi_n\ \therefore\ \psi\big)$,
means $\varphi_1,\ldots,\varphi_n \models_\PSA \psi$.

\end{itemize}

\begin{remark} 
Since this paper 
only deals with algebras from $\PSA$, 
the prefix and subscript $\PSA$, etc., {\em will usually be omitted}.
\end{remark}

{\bf Basic formulas}  are {\bf equations} $p(\bfx) = q(\bfx)$ and 
{\bf negated equations} $p(\bfx) \neq q(\bfx)$. 
They suffice to express a variety of propositions about sets, 
including the famous Aristotelian categorical propositions.\footnote{
Schr\"oder used $\neq 0$ to translate particular propositions into symbolic
form in his {\em Algebra der Logik} (p.~93 in Vol. II). 
In this work he also `proved' that Boole's efforts to translate 
particular propositions by equations (using the infamous symbol $V$) 
must fail (pp.~91-93 in Vol. II). 
Yet in $\S$\ref{sec on V} of this paper,
the reader will find that a slight variation on Boole's use of $V$ indeed 
works {\em in the context of valid arguments}. 
And in $\S$\ref{sec on Velim} we will find that it works {\em in the context
of elimination} as well.
}
For example, the assertion `All $x$ is $y$' is expressed by $x = xy$, or equivalently,
$xy' = 0$, since for
any interpretation $I$ in a power-set algebra $\bPS(U)$, one has, setting
$A := I(x)$ and $B := I(y)$, `All $A$ is $B$' holding iff $A\subseteq B$, and this
holds iff $A = A B$, or equivalently, $AB' = \O$.
The following table gives a sampler of propositions that can be expressed
by a basic formula:
\begin{quote}
\begin{tabular}{l  | l  | l}
Proposition&Basic Formula&Alternative\\
\hline
$x$ is empty & $x = 0$\\
$x$ is not empty & $x \neq 0$\\
\hline
All $x$ is $y$ & $xy' = 0$ & $x = xy$\\
No $x$ is $y$ & $xy = 0$& $x = xy'$\\
Some $x$ is $y$ & $xy \neq 0$\\
Some $x$ is not $y$ & $xy' \neq 0$\\
\hline
$x$ and $y$ are empty & $x\plus  y = 0$\\
$x$ and $y$ are disjoint & $xy = 0$\\
$x$ is empty and $y$ is the universe & $x \plus   y' = 0$\\
$x$ or $y$ is not empty &$x\plus  y \neq 0$\\
etc.&
\end{tabular}
\end{quote}
However some simple relationships among sets cannot be expressed
by basic formulas, for example, `$x$ is empty or $y$ is empty', 
`$x$ is empty implies $y$ is empty', 
 `there are at least 2 elements in the universe', etc.

Propositions are usually formulated in ordinary language, with a few
symbols, like `All $S$ is $P$'. It is not so easy to precisely describe
all the ordinary language statements that qualify as propositions
about sets. To get around this awkward situation we simply
define our {\bf domain of propositions} about sets to be
 all propositions $\pi(\bfx)$
which can be expressed by basic formulas $\beta(\bfx)$. 
Then it is automatic that a list of propositional premisses
$$
\pi_1(\bfx),\ldots, \pi_n(\bfx) 
$$
can be expressed by a list of basic formulas
$$
\beta_1(\bfx),\ldots, \beta_n(\bfx);
$$
and a propositional argument
$$
\pi_1(\bfx),\ldots, \pi_n(\bfx) \ \therefore\ \pi(\bfx)
$$
can be expressed by a basic-formula argument
$$
\beta_1(\bfx),\ldots, \beta_n(\bfx)\ \therefore\ \beta(\bfx).
$$

Boole focused on two themes in his algebra of logic, namely 
given a  list 
$\pi_1(\bfx,\bfy)$,\ldots, $\pi_n(\bfx,\bfy) $ 
of propositional premisses:
\begin{thlist}
\item
how to find the `complete' result $\pi(\bfy)$ 
of eliminating the variables $\bfx$
from the propositional premisses; and
\item
how to express a variable $x_i$ in terms of the other variables, given
the propositional premisses.
\end{thlist}
In our modern version, the propositional premisses correspond exactly 
to basic-formula premisses
$\beta_1(\bfx,\bfy),\ldots, \beta_n(\bfx,\bfy)$,
and the propositional themes are clearly equivalent to the basic-formula
themes:
\vspace{-.5cm}
\begin{thlist}
\item
how to find the `complete' result $\beta(\bfy)$ of eliminating 
the variables $\bfx$
from the basic-formula premisses; and
\item
how to express a variable $x_i$ in terms of the other variables, given
the basic-formula premisses.
\end{thlist}

In 1854 Boole \cite{Boole-1854} presented a General Method for tackling these questions about 
propositions, a method that used only equations, thus avoiding the use of negated 
equations. (For a summary of Boole's General Method in modern notation, see
\cite{Burris-SEP}.)
We can parallel essentially all of Boole's General Method in the modern framework 
described above, with the advantage that neither the translations nor the 
methods are suspect.

\section{Axioms and Rules of Inference} \label{sec ARI}

The \textit{laws} or \textit{axioms} for the three set operations
are as follows, where $p$, $q$ and $r$ are any three terms---
these basic formulas are satisfied by $\PSA$:
\small
$$
\begin{array} {l | l | l}
 p\plus    p =  p  &
 p \cdot  p =  p & \text{Idempotent Laws}\\
  p\plus    0 =  p  &
 p \cdot   0 =  0 & \text{0 Laws}\\
   p\plus    1 =  1  &
 p \cdot   1 =  p & \text{1 Laws}\\
 p \plus    q =   q \plus    p  &
 p \cdot   q =  q \cdot   p &
 \text{Commutative Laws}\\
 p \plus   (  q \plus    r ) = ( p \plus     q ) \plus    r &
 p \cdot  (  q \cdot    r )  = ( p \cdot     q ) \cdot    r &
 \text{Associative Laws}\\
 p \plus   p\cdot    q =  p &
 p\cdot   ( p \plus    q) =  p &
 \text{Absorption Laws}\\
 p \plus    q \cdot   r = ( p \plus    q)\cdot   ( p \plus    r) &
 p \cdot  ( q \plus    r) =  p \cdot   q \plus    p\cdot    r &
 \text{Distributive Laws}\\
 p \plus    p' =  1 &
 p \cdot   p' =  0 &
 \text{Complement Laws} \\
 (p\plus  q)' = p'\cdot   q' &
 (p\cdot   q)' = p' \plus   q' &
 \text{De Morgan Laws}\\
\multicolumn{3}{l}{\text{and there is one inequality}}\\
1 \neq 0&&
\text{Non-empty Universe}
\end{array}
$$
\normalsize
(These axioms are somewhat redundant.)

The usual \textit{equational rules of inference}, 
where $p,q,r,s$ are any four terms, are:\\
\small
\begin{itemize}
\item
the reflexive, symmetric and transitive rules for equality 
\bigskip
\item
$\displaystyle \frac{p=q}{ p' = q'}$ 
 \quad\hfill (Complement of equals)
\bigskip
\item
$\displaystyle \frac{p=q,\ r=s}{ p\plus  r = q\plus  s}$ 
 \quad\hfill (Union of equals)
\bigskip
\item
$\displaystyle \frac{p=q,\ r=s}{ p \cdot r = q \cdot  s}$ 
 \quad\hfill (Intersection of equals)
\bigskip
 \end{itemize}
 \normalsize
 
 [NOTE: The last three rules are equivalent to the {\em replacement} rule.]

\subsection{A Standard Form}

Every equation $p=q$ can be put in a {\bf standard form} $r=0$.
\begin{lemma} [{\bf Standard Form}] \label{std form}
An equation $p(\bfx)=q(\bfx)$ is equivalent to an equation in the form 
$r(\bfx)=0$, namely let $r(\bfx) = p(\bfx)\, \triangle\, q(\bfx)$, the symmetric
difference of $p(\bfx)$ and $q(\bfx)$, which is defined by:
$$
p\, \triangle\, q\ :=\ p\cdot q'\ \plus\ p' \cdot q. 
$$
\end{lemma}

\section{Constituents}

Boole introduced constituents to provide a basis for expanding terms.

\begin{definition} Given a list of variables $\bfx := x_1,\ldots,x_k$, the 
$2^k$ {\bf consitituents} of $\bfx$ are the following terms:

$\begin{array}{l}
x_1x_2\cdots x_k\\
x_1'x_2\cdots x_k\\
\quad\vdots\\
x_1'x_2'\cdots x_k'.
\end{array}$\\
These are called the {\bf $\bfx$-constituents}.
A useful notation for referring to them is as follows.
For $\sigma\in 2^{\widetilde{k}}$, that is, for $\sigma$ a mapping from 
$ \{1,\ldots,k\}$ to $\{0,1\}$,
 let
$$
\cC_\sigma(\bfx)\ :=\ \cC_{\sigma_1}(x_1)\cdots \cC_{\sigma_k}(x_k),
$$
where $\cC_1(x_j) := x_j$ and $\cC_0(x_j) := x_j' $ and $\sigma_i := \sigma(i)$.
\end{definition}
Thus, for example, with $k=5$ and $\sigma = 01101$, 
$\cC_\sigma(\bfx) = x_1'x_2x_3x_4'x_5$.

\begin{lemma} \label{basic constit}
For $i,j \in \{0,1\}$, $\PSA$ satisfies
\begin{eqnarray}
\cC_{i}(j)&=& 1\quad \text{if }\  j = i  \label{i,j equal} \\
\cC_{i}(j)&=& 0 \quad \text{if }\  j \neq i. \label{i,j nequal} 
\end{eqnarray}
\end{lemma}
\begin{proof}
From the definition of $\cC_i(x)$.
\end{proof}

\begin{lemma}  \label {constit val}
For $\sigma, \tau \in 2^{\wtk}$, and $p(\bfx)$ a term, $\PSA$ satisfies
\begin{align}
p(\sigma)=1 \ \text{or}\ p(\sigma)=0 \label{t(sig)}\\
 \cC_\sigma(\tau) = 1 &\quad \text{if } \sigma = \tau\label{s,t equal}\\
 \cC_\sigma(\tau) = 0 &\quad \text{if } \sigma \neq \tau . \label{s,t nequal}
\end{align}
\end{lemma}

\begin{proof}
The first item is proved by induction on the term $p(\bfx)$.

For the second item,
 suppose $\sigma = \tau$. 
Then $\cC_{\sigma_j}(\tau_j)\ =\ 1$  for $j =1,\ldots, k$, by \eqref{i,j equal}, so
$$
\cC_\sigma(\tau) \ :=\   \cC_{\sigma_1}(\tau_1)\cdots \cC_{\sigma_{k}}(\tau_{k}) \ =\  1.
$$

Finally, suppose that  $\sigma \neq \tau$. 
For some $i$ we have
$\sigma_i \neq \tau_i$.
Then $\cC_{\sigma_i}(\tau_i) \ =\ 0$, by \eqref{i,j nequal},
so
$$
\cC_\sigma(\tau) \ :=\   
\cC_{\sigma_1}(\tau_1)\cdots \cC_{\sigma_{k}}(\tau_{k}) \ =\  0.
$$

\end{proof}

\begin{lemma} \label {prop constit}
For $\sigma, \tau\in 2^\wtk$, $\PSA$ satisfies
\begin{align} 
\cC_{\sigma}(\bfx)\cdot   \cC_{\tau}(\bfx) &= 
\cC_{\sigma}(\bfx) \quad & \text{if }\sigma = \tau \label{C idemp}\\
 \cC_{\sigma}(\bfx)\cdot \cC_{\tau}(\bfx) &= 0 \quad &\text{if }\sigma \neq \tau \label{perp}\\
\summe_{\sigma\in 2^{\wtk}} \cC_{\sigma}(\bfx)& = 1.  \label{partition}
\end{align}
\end{lemma}

\begin{proof}
For \eqref{C idemp} and \eqref{perp}, use the fact that one can derive
\begin{eqnarray*}
 \cC_{\sigma}(\bfx)\cdot \cC_{\tau}(\bfx) 
 &=&\Big( \prodkt_{j \in \wtk} \cC_{\sigma_j}(x_j)\Big) \ \cdot\ 
 \Big(  \prodkt_{j \in \wtk} \cC_{\tau_j}(x_j)\Big)\\
 &=& \prodkt_{j \in \wtk} \Big( \cC_{\sigma_j}(x_j)\cdot  \cC_{\tau_j}(x_j) \Big).
\end{eqnarray*}
If $\sigma \neq \tau$ then for some $j$ one has $\{\sigma_j, \tau_j\} = \{0,1\}$, 
so 
 $\big\{\cC_{\sigma_j}(x_j) ,\cC_{\tau_j}(x_j) \big\} = \{x_j,x_j'\}$, leading to 
$\cC_{\sigma}(\bfx)\cdot \cC_{\tau}(\bfx) = 0$.

If $\sigma = \tau$ then for each $j$ one has 
$$
\cC_{\sigma_j}(x_j)\cdot  \cC_{\tau_j}(x_j)\  
=\  \cC_{\sigma_j}(x_j)\cdot  \cC_{\sigma_j}(x_j)\  =\  \cC_{\sigma_j}(x_j),
 $$
and thus
$$
\cC_{\sigma}(\bfx)\cdot \cC_{\tau}(\bfx)\  
=\  \prodkt_{j \in \wtk} \cC_{\sigma_j}(x_j)\ :=\  \cC_{\sigma}(\bfx).
$$

For \eqref{partition}, use the fact that one has
$$
1 \ =\ \prodkt_{j \in \wtk} \Big( x_j \plus   x_j' \Big)\ 
=\  \prodkt_{j \in \wtk} \Big( \cC_1 (x_j) \plus   \cC_{0} (x_j) \Big) .
$$
Expanding the right side gives the desired expression of $1$ as the 
union of all the $\bfx$-constituents.
\end{proof}

\begin{lemma} \label{term x constit}
Given a term $t(\bfx,\bfy)$ and an $\bfx$-constituent $\cC_\sigma(\bfx)$, 
 $\PSA$ satisfies
$$
t(\bfx,\bfy)\cdot \cC_\sigma(\bfx)\ =\ t(\sigma,\bfy)\cdot \cC_\sigma(\bfx).
$$
\end{lemma}

\begin{proof}
By induction on the term $t(\bfx,\bfy)$.
\end{proof}

\subsection{Reduction Theorem}

Every list of equations can be reduced to a single equation.
\begin{theorem} [{\bf Reduction}]  \label{reduction}
A list of equations $p_1(\bfx) = 0,\ldots,p_n(\bfx)=0$ is equivalent
to the single equation $p_1(\bfx) \plus  \cdots \plus  p_n(\bfx)=0$.
\end{theorem}

\begin{proof}
The direction $(\Rightarrow)$ is clear. 
For the direction $(\Leftarrow)$,
multiply $p_1(\bfx) \plus  \cdots \plus  p_n(\bfx)=0$ by any $p_i(\bfx)$ and use
an absorption law.
\end{proof}

\begin{remark} 
Reduction was a key step for Boole because his 
Elimination Theorem only applied to a single equation, not to a
list of equations. He had to use a more complicated expression than
Theorem \ref{reduction}
for his system---he developed several forms for reduction,
the main one being $p_1(\bfx)^2 +  \cdots +  p_n(\bfx)^2=0$
{\rm (see \cite{Boole-1854}, p. 121)}.
\end{remark}

\subsection{Expansion Theorem}

Any term $t(\bfx,\bfy)$ can be expanded as a `linear' combination
 of $\bfx$-constituents, with
coefficients that are terms in the variables $\bfy$.
\begin{theorem} [{\bf Boole's Expansion Theorem}] \label{expansion thm}
Given a term $t(\bfx,\bfy)$,  $\PSA$ satisfies
$$
t(\bfx,\bfy)\ =\ \summe_{\sigma \in 2^{\wtk}} t(\sigma,\bfy)\cdot \cC_\sigma(\bfx).
$$
In particular, 
$$
t(x,\bfy)\ =\  t(1,\bfy) \cup t(0,\bfy).    
$$
\end{theorem}
\begin{proof}
From the third item of Lemma \ref{prop constit}, and Lemma \ref{term x constit}, 
we have
\begin{eqnarray*}
t(\bfx,\bfy)
&=& \summe_{\sigma \in 2^{\wtk}} t(\bfx,\bfy)\cdot  \cC_\sigma(\bfx)\\
&=& \summe_{\sigma \in 2^{\wtk}} t(\sigma,\bfy)\cdot  \cC_\sigma(\bfx).
\end{eqnarray*}
\end{proof}

A special case that occurs frequently is when one expands about all the
variables in the term---the result, a union of consitituents, is  
called {\em the full expansion} of the term; it is also known as
 {\em the disjunctive normal form} 
of the term.
\begin{corollary} \label{DNF}
Given a term $t(\bfx)$,  $\PSA$ satisfies
\begin{eqnarray*}
  t(\bfx) &=& \summe_{\sigma \in 2^{\wtk}} t(\sigma)\cdot \cC_\sigma(\bfx)
\ = \  \begin{cases}
0&\text{if }\quad \PSA \models t(\bfx)=0\\
\summe_{\substack{\sigma \in 2^{\wtk}\\t(\sigma)\neq 0}}  \cC_\sigma(\bfx) 
& \text{otherwise}.
\end{cases}
\end{eqnarray*}

\end{corollary}
%



%
\begin{proof}
The first equality is from Theorem \ref{expansion thm}. 
For the second, note that, in the case of a full expansion, each coefficient  
$t(\sigma)$ is either 0 or 1, by \eqref{t(sig)}. 
\end{proof}

\subsection{More on Constituents}

\begin{definition}
Given a term $t(\bfx)$,
 the {\bf set of constituents of} $t(\bfx)$
is
$$
\cC\big(t(\bfx)\big)\ :=\ \big\{ \cC_\sigma(\bfx) : t(\sigma) \neq 0 \big\}.
$$
\end{definition}
With this definition we have an easy consequence of Corollary \ref{DNF}.
\begin{corollary}\label{test eq}
$\PSA \models s(\bfx)=t(\bfx)\ $ iff 
$\ \cC\big(s(\bfx)\big) = \cC\big(t(\bfx)\big)$.
\end{corollary}

The next result makes the expressive power of an equation $p(\bfx)=0$
 clear---all
it says is that the constituents of $p(\bfx)$ are empty.
\begin{corollary}\label{expr power}
An equation $t(\bfx)=0$ is $\PSA$-equivalent to the (conjunction of the) set of equations
$$
\{0=0\}\ \cup\ \big\{\cC_\sigma(\bfx) = 0 : 
\cC_\sigma(\bfx) \in \cC\big(t(\bfx)\big)\big\}.
$$
\end{corollary}
The adjunction of $\{0=0\}$ is needed for the case
 $\PSA \models t=0$.

The next lemma gives a simple calculus for working with constituents of terms.

\begin{lemma} \label{constit prop}
Let $\cC(\bfx)$ be the set of $\bfx$-constituents. 
\begin{thlist}
\item
$\cC(0) = \O$ and $\cC(1) = \cC\bfx$.
\item
$\cC\big(t_1 \plus   \cdots \plus   t_{n}\big)\ =\ \cC\big(t_1\big) \plus  
\cdots \plus  \cC\big(t_{n}\big)$
\item
$\cC\big(t_1  \cdots t_{n}\big)\ =\ \cC\big(t_1\big)  \cdots  \cC\big(t_{n}\big)$
\item
$\cC(t') \ =\ \cC(t)'$.
\end{thlist}
\end{lemma}
\begin{proof} (Routine.) \end{proof}

\section{Valid Basic-Formula Arguments}

The goal of finding general conditions under which arguments
$\varphi_1,\ldots,\varphi_n\ \therefore\ \varphi$ are valid does not seem 
to have been part of the algebra or logic of the 1800s. 
One was not so much interested in devising a test to see if one had 
found a consequence of a set of premisses; rather the focus was on forging 
methods for actually finding consequences from the premisses, preferably
 the most general consequences.
The topics that interested Boole, and later Schr\"oder, in the algebra of logic 
were {\em elimination} and {\em solution}---we will discuss those later, 
in $\S$\ref{sec E and S}.

\begin{definition}
Given a term $p(\bfx)$, let $C_1(\bfx),\ldots,C_m(\bfx)$ \rm{(}where $m = 2^k$\rm{)} be a listing of the        
$\bfx$-constituents.
Define the universe $U_p$ and an interpretation $I_p$ into $\bPS(U_p)$ by:
\begin{eqnarray*} 
U_p &:=& \big\{ i \in \{1, \dots , m\} : C_i(\bfx) \notin \cC(p) \big\}\\                
I_p(x_\ell)&:=& A_\ell \ :=\  \big\{i \in U_p : C_i(\bfx) \in \cC(x_\ell)\big\}, \text{ for }  \ell=1, \dots, k.
\end{eqnarray*}
Note that $C_i(\bfx) \in \cC(x_\ell)$ means that $x_\ell$, and not $x_\ell'$,
 appears in $C_i(\bfx)$.
\end{definition}

  \begin{example}
  Let  $\bf(x):=x_1, x_2$ and $p(x_1,x_2)$ be $x_1x_2 \plus   x_1'  x_2'$.
  List the four $\bfx$-constituents:
  $$
  \begin{array}{r c l @{\qquad} r c l}
  C_1(x_1,x_2) &:=& x_1  x_2 &
  C_2(x_1,x_2) &:=& x_1  x_2'\\
 C_3(x_1,x_2) &:=& x_1'  x_2&
  C_4(x_1,x_2) &:=& x_1'  x_2'.
  \end{array}
  $$
  Then  $\cC(x_1) = \{ \cC_1, \cC_2\}$, $\cC(x_1) = \{ \cC_1, \cC_3\}$, and
  $\cC\big(p(x_1,x_2)\big) = \big\{ \cC_1, \cC_4\}$, so we have $U_p := \{2,3\}$, 
   $I_p(x_1) = \{2\}$, $I_p(x_2) = \{ 3 \}$, $I_p(\cC_1(\bfx)) =I_p(x_1) \cap I_p(x_2)=\{2\} \cap \{3\}=\O$,
 $I_p(\cC_2(\bfx)) = \{2\}$, $I_p(\cC_3(\bfx)) = \{3\}$, $I_p(\cC_4(\bfx)) = \O$.  Observe that if $(\cC_i(\bfx)) \in \cC\big(p(\bfx)\big)$, then $I_p(\cC_i(\bfx)) = \O$, and if $(\cC_i(\bfx)) \notin \cC\big(p(\bfx)\big)$, then $I_p(\cC_i(\bfx)) = \{i\}$.
 
  \end{example}
  
The following lemma generalizes the above example.

\begin{lemma}
$I_p$ interprets the $C_i(\bfx)$ in $U_p$ as follows:
\begin{equation} \label{Ip}
I_p\big(C_i(\bfx)\big)\ :=\ C_i(\bA) \ = \ \begin{cases}
\O&\text{if }\ C_i(\bfx) \in \cC(p)\\
\{i\}&\text{if }\ C_i(\bfx) \notin \cC(p).
\end{cases}
\end{equation}
Thus $I_p\big(C_i(\bfx)\big) \neq \O$ iff $C_i(\bfx)$ is not a constituent
of $p(\bfx)$.
\end{lemma}

\begin{proof}
We have
\begin{eqnarray*}
j \in I_p\big(C_i(\bfx)\big)
&\Leftrightarrow& 
\bigwedge_{\ell \in \wtk}\Big[ C_i(\bfx) \in \cC(x_\ell) \Leftrightarrow j \in A_\ell\Big]\\
&\Leftrightarrow& 
\bigwedge_{\ell \in \wtk}\Big[ C_i(\bfx) \in \cC(x_\ell) 
\Leftrightarrow \big( j \in U_p \ \wedge\ C_j(\bfx)\in \cC(x_\ell)\big)\Big]\\
&\Leftrightarrow& 
\big( j \in U_p\big)\ \wedge\ 
\bigwedge_{\ell \in \wtk}\Big[ C_i(\bfx) \in \cC(x_\ell) 
\Leftrightarrow  C_j(\bfx)\in \cC(x_\ell)\Big]\\
&\Leftrightarrow& 
\big( j \in U_p\big)\ \wedge\ 
\big(C_i(\bfx) = C_j(\bfx)\big)\\
&\Leftrightarrow& 
\big( j \in U_p\big)\ \wedge\ 
\big(i=j\big)\\
&\Leftrightarrow& 
\Big( C_i(\bfx) \notin \cC\big(p(\bfx)\big)\Big)\ \wedge\ 
\big(j=i\big)
\end{eqnarray*}

\end{proof}

%
  %
\begin{lemma} \label{using Ip}
Given terms $p(\bfx)$ and $q(\bfx)$, one has
\begin{thlist}
\item
$I_p\big(\cC_\sigma(\bfx)\big) = \O\ $ iff $\ \cC_\sigma(\bfx)\in\cC\big(p(\bfx)\big)$, 
for $\sigma\in 2^\wtk$.
\item
$\cC\big(q(\bfx)\big) \subseteq \cC\big(p(\bfx)\big)$
iff $I_p(q) = \O$.
\item
The equational argument 
	$$p(\bfx)=0 \ \therefore\ q(\bfx)=0$$
 is valid iff
	$$\cC\big(q(\bfx)\big) \subseteq \cC\big(p(\bfx)\big).$$
\end{thlist}
\end{lemma}
\begin{proof}
(a) follows from \eqref{Ip}, and (b) from (a) and Corollary 3.9. 
The direction $(\Leftarrow)$ of (c) follows from Corollary \ref{expr power},
For the direction $(\Rightarrow)$ of (c), assume 
$\Valid\big[p(\bfx)=0 \ \therefore\ q(\bfx)=0\big]$. By Lemma \ref{using Ip}(b),
 $I_p(p)=0$, thus we must have $I_p(q)=0$. This gives 
 $\cC\big(q(\bfx)\big) \subseteq \cC\big(p(\bfx)\big)$, again by 
 Lemma \ref{using Ip}(b).
\end{proof}

\subsection{Equational Arguments}
The next result says that an equational
argument is valid  iff the constituents of the conclusion are
among the constituents of the premisses.
\begin{theorem} [{\bf Equational Arguments}] \label{constit and validity}
The equational argument
$$
p_1(\bfx)=0,\ \ldots,\  p_m(\bfx)=0\ \therefore\ p(\bfx)=0
$$
is valid  iff
$$
\cC\big(p(\bfx)\big)\ \subseteq\ 
\cC(p_1\big(\bfx)\big) \plus  \cdots \plus  \cC\big(p_m(\bfx)\big).
$$
\end{theorem}
\begin{proof}
This follows from Lemma \ref{constit prop}(b) and Lemma \ref{using Ip}(c),
since the premisses can be reduced to the single equation
$p_1(\bfx) \plus  \cdots\plus   p_m(\bfx)=0$, by Theorem \ref{reduction}.
\end{proof}

\subsection{Equational Conclusion}
Some basic-formula arguments are rather trivially valid
 because it is not 
possible to make all the premisses true, under any interpretation.
A simple example would be $x=0, x\neq 0 \ \therefore\ \beta$.
This argument is  valid, but not very interesting. 
The next theorem says that positive conclusions only require
positive premisses, provided the premisses are satisfiable.
\begin{theorem} [{\bf Equational Conclusion}]  \label{eq concl}
Suppose the list 
\begin{equation} \label{nonTriv}
	p_1(\bfx)=0,\ \ldots,\ p_m(\bfx)=0,\ q_1(\bfx)\neq 0,\ \ldots,\ q_n(\bfx) \neq 0
\end{equation}
of basic formulas is satisfiable. 
Then the basic-formula argument
\begin{equation} \label{PNP}
	p_1(\bfx)=0,\ \ldots,\ p_m(\bfx)=0,\ q_1(\bfx)\neq 0,\ \ldots,\ q_n(\bfx) \neq 0
\ \therefore\ p(\bfx)=0
\end{equation}
is  valid iff the equational argument
\begin{equation}\label{PP}
	p_1(\bfx)=0,\ \ldots,\ p_m(\bfx)=0\ \therefore\ p(\bfx)=0
\end{equation}
is  valid.
\end{theorem}
\begin{proof}

The direction $\eqref{PP} \Rightarrow \eqref{PNP}$ is trivial. 
So suppose \eqref{PNP} is  valid.
First let us replace the equational premisses in \eqref{PNP} with a single equation
$p_0(\bfx) = 0$, giving the argument
\begin{equation} \label{P0NP}
	p_0(\bfx)=0,\ q_1(\bfx)\neq 0,\ \ldots,\ q_n(\bfx)\neq 0
\ \therefore\ p(\bfx)=0,
\end{equation}
 where $p_0(\bfx) := p_1(\bfx)\plus  \cdots\plus  p_m(\bfx)$. 
 The arguments
 \eqref{PNP} and \eqref{P0NP} are both  valid or both  invalid.
 
By Corollary \ref{test eq} and Lemma \ref{using Ip}(a), 
the interpretation $I_{p_0}$ makes $p_0(\bfx)=0$ true.
From the satisfiability of \eqref{nonTriv},
 it follows that
$$
	p_0(\bfx)=0,\ q_1(\bfx)\neq 0,\ \ldots,\ q_n(\bfx)\neq 0
$$
is satisfiable.  Then by Lemma \ref{using Ip}(b), each $\cC(q_j)$
has a constituent that is not in $\cC(p_0)$. 
Thus $I_{p_0}$ makes some constituent in each $\cC(q_j)$ 
non-empty, and thus it makes each $q_j(\bfx)\neq 0$ true. 
From this it follows that the interpretation $I_{p_0}$ makes all the 
premisses of \eqref{PNP} true. 
Since we have assumed \eqref{PNP} is a valid argument, it follows
that $I_{p_0}$ makes $p(\bfx)=0$ true, thus $I_{p_0}\big(p(\bfx)\big) = \O$.
By Lemma \ref{using Ip} (b), it follows that $\cC(p) \subseteq \cC(p_0)$; 
consequently, by Lemma \ref{using Ip}(c), 
the argument $p_0(\bfx)=0 \therefore p(\bfx)=0$ is valid. 
Thus \eqref{PP} is valid.

\end{proof}

\subsection{Negated-Equation Conclusion}
Now we turn to the case when the conclusion is a negated equation. 
Perhaps surprisingly, such an argument reduces in a simple manner 
to a disjunction of equational arguments.
We assume the equational premisses have already been reduced to a single 
equation $p_0(\bfx)=0$.

\begin{theorem} \label{NegRed}
Consider the following assertions:
\begin{align}
&\Valid \big( p_0(\bfx) = 0,\  q_1(\bfx)\neq 0,\ \ldots,\  q_n(\bfx) \neq 0 
 \ \therefore\ q(\bfx) \neq 0 \big)
 \label{PNN1}\\
&\Valid\big( p_0(\bfx) = 0,\ q_j(\bfx) \neq 0 \ \therefore\ q(\bfx) \neq 0  \big)
 \label{PNN2}\\
&\Valid\big( p_0(\bfx) = 0,\ q(\bfx)= 0\ \therefore\ q_j(\bfx) = 0 \big)
 \label{PNN3}\\
& \cC\big(q_j(\bfx)\big)\ \subseteq\ \cC\big(p_0(\bfx)\big)\ \plus
\ \cC\big( q(\bfx)\big). \label{PNN4}
\end{align}

Then 
\begin{thlist}
\item
For each $j$, 
\eqref{PNN2} holds iff  \eqref{PNN3} holds.

\item
For each $j$, 
\eqref{PNN3} holds iff \eqref{PNN4} holds.


\item
\eqref{PNN1} holds iff for some $j$, \eqref{PNN2} holds.

\end{thlist}
\end{theorem}

\begin{proof}
Item (a) follows from simple propositional logic, namely the propositional formula 
$(P \wedge \neg Q) \rightarrow \neg R$
is equivalent to $(P \wedge R) \rightarrow Q$.

Item (b) follows from Theorem \ref{constit and validity}.

The direction $(\Leftarrow)$ of (c) clearly holds. So it only remains to show 
that if \eqref{PNN1} holds one has \eqref{PNN2} holding for some $j$---we will 
show the contrapositive.

 Suppose for every $j$, \eqref{PNN2} is false.
Then for each $j$, \eqref{PNN4} is false, by (a) and (b).
Let $\widehat{p} := p_0 \plus   q$. 
From the failure of \eqref{PNN4} for each $j$, one has 
$\cC(q_j) \nsubseteq \cC(\widehat{p})$ for each $j$. 
Consequently, by Lemma \ref{using Ip}(b), the interpretation
$I_{\widehat{p}}$ makes $\widehat{p}=0$ true, but $q_j=0$ false, for each $j$.
This means $I_{\widehat{p}}$ makes the premisses of the argument in 
\eqref{PNN1} true, but the conclusion false. 
Thus \eqref{PNN1} does not hold.

\end{proof}

Combining the above theorems, we see that the study of valid basic-formula
arguments reduces in a simple manner to the study of valid equational
arguments, which in turn reduces to comparing constituents of the terms
involved in the arguments. 

\subsection{Using Boole's $V$ in Valid Arguments}\label{sec on V}

So far we have adopted Schr\"oder's translation of particular propositions,
using $\neq 0$. Boole did not do this, but rather tried to use an equational
translation. Consider the proposition `Some $x$ is $y$'. In 1847 he used
$V = xy$ as his primary translation (see \cite{Boole-1847}, p.~20), 
where $V$ is a new idempotent 
symbol. In 1854 (see \cite{Boole-1854}, pp.~61-64) 
he changed the translation to $V\cdot x = V\cdot y$. Neither 
translation seems fully capable of doing what Boole claimed, although
the first seems closer to achieving his goals. There is a simple
intermediate translation, namely $V = V\cdot xy$, or equivalently,
$V \cdot (xy)'=0$, that works much better. We say it is an intermediate
translation because 
$V=xy \ \Rightarrow\ V=V \cdot xy\ \Rightarrow\ V\cdot x=V\cdot y$.

Let us call an equation of the form $V\cdot p(\bfx) = 0$ a {\bf $V$-equation}.
Next we show how the two-way translation 
$$
p\neq 0\quad \rightleftharpoons  \quad V\cdot p' = 0
$$
between negated equations and $V$-equations
 can be used in the study of valid
arguments, to fulfill Boole's goal of
a viable {\em equational} translation of particular statements.

First we look at the case of a single negative premiss, where we use
a simple fact about sets, namely
\begin{equation}\label{simple fact}
A\subseteq B\plus  C\quad\text{iff}\quad C' \subseteq B \plus  A'.
\end{equation}
This follows from noting that both sides are equivalent to $AB'C' = \O$.

\begin{theorem} \label{one V}
The following are equivalent:
\begin{thlist}
\item 
$$
\Valid\big( p = 0,\, q_0\neq 0  \ \therefore q \neq 0 \big)
$$
\item 
$$
\Valid\big( p = 0, \, V \cdot q_0'  = 0 \ \therefore V \cdot q' = 0 \big).
$$
\end{thlist}
\end{theorem}
\begin{proof}
By simple propositional logic reasoning, (a) is equivalent to
\begin{equation}\label{1V1}
\Valid \big( p = 0,\, q = 0 \ \therefore\ q_0 = 0 \big).
\end{equation}
By Theorem \ref{constit and validity}, assertion \eqref{1V1} is equivalent to
\begin{equation}\label{1V2}
 \cC(q_0) \ \subseteq\  \cC(p) \plus   \cC(q),
\end{equation}
 which, in view of \eqref{simple fact} and Lemma \ref{constit prop}(d), we can rewrite as
\begin{equation}\label{1V3}
\cC(q')\  \subseteq\  \cC(p) \plus   \cC(q_0').
\end{equation}
Looking at $(V,\bfx)$-constituents,
\eqref{1V3} is equivalent to
\begin{eqnarray} \label{1V4}
\cC(V \cdot q') 
&\subseteq& \cC(p) \plus   \cC(V \cdot  q_0'),
\end{eqnarray}
which, by Theorem \ref{constit and validity},
is equivalent to
$$
\Valid\big( p =0,\, V \cdot q_0' = 0 \ \therefore V \cdot q' = 0 \big).
$$
\end{proof}

When looking at the equivalence of (a) and (b) in this theorem it is
easy to jump to the conclusion that somehow one has been able
to replace each negated equation with an {\em equivalent} equation involving a
new symbol $V$; that, for example, the assertion ``Some $x$ is $y$'' is 
fully expressed by the equation $V=V \cdot (xy)$, or equivalently, $V \cdot (xy)'=0$.
If we think of $V$ as standing for `something', then these two equations
can be viewed as saying ``Something is in both $x$ and $y$''. That would
be useful as a mnemonic device, but the reality is that $xy \neq 0$ and
$V \cdot (xy)'=0$ do not express the same thing, that is, they are not equivalent.

One heuristic behind the use of a new constant $V$ is that 
one can express $t \neq 0$ with a formula
$$
(\exists V) \big(V \neq 0 \text{ and } V\subseteq t\big)
$$
which is equivalent to 
$$
(\exists V) \big(V \neq 0 \text{ and } V = V \cdot  t\big)
$$
as well as
$$
(\exists V) \big(V \neq 0 \text{ and } V \cdot  t' = 0 \big).
$$
When reasoning with such a formula it would be usual to say
``Choose such a $V$'', giving
$$
V \neq 0 \text{ and } V \cdot  t' = 0.
$$
This formula is, of course, not equivalent to the original $t\neq 0$, but it
implies the latter. When one drops the formula $V\neq 0$, then also this
implication fails, that is, neither of $t\neq 0$ and $V \cdot t' = 0$ implies the other.
Thus the equivalence of (a) and (b) in Theorem \ref{one V} is certainly
 not due to a simple replacement of basic formulas by equivalent formulas. 
 The fact that (a) and (b) are equivalent is based on the magic
 of the global interaction of formulas in an argument---it is something that
 one does not expect to be true, but it might be true, and through the 
 curiosity of exploration one discovers a proof that indeed the two are equivalent.

\begin{theorem} \label{Boole's method}
The following are equivalent:
\begin{thlist}
\item
$$
\Valid\big( p = 0,\ q_0\neq 0,\ \ldots,\ q_{n-1} \neq 0  \ \therefore q \neq 0 \big).
$$
\item
For some $j$, 
$$
\Valid\big( p = 0,\ q_j \neq 0  \ \therefore  q \neq 0 \big).
$$
\item
For some $j$, 
$$
\Valid\big( p = 0,\ V_j \cdot q_j' = 0  \ \therefore V_j \cdot q' = 0 \big).
$$

\item 
For some $j$, 
$$
\Valid\big( p = 0, \ V_0\cdot  q_0' = 0,\ \ldots,\ V_{n-1}\cdot  q_{n-1}' = 0 
\ \therefore V_j\cdot  q' = 0 \big).
$$
\end{thlist}
\end{theorem}

\begin{proof}

By Theorem \ref{NegRed}(c), item (a) is equivalent to (b); and         
 by Theorem \ref{one V}, (b) is equivalent to (c).

Clearly {(c) implies (d)}. Now suppose (d) holds, and choose a $j$
such that the indicated argument is valid.
By setting all $V_i$, $i\neq j$, equal to $0$, we have {(d) implies (c)}.

\end{proof}

\section{Elimination and Solution} \label{sec E and S}

\subsection{A Single Equation}
The following gives Schr\"oder's version of Boole's results on elimination 
and solution
(p.~447 in \cite{Schr} Vol. I).
The elimination condition is the same as Boole's, but the solution is much simpler, 
thanks
to working with power-set algebras instead of Boole's system.
\begin{theorem}[{\bf Elimination and Solution Theorem}] \label{E and S}
Given a term $p(x,\bfy)$, 
the equation $p(x,\bfy)=0$ is $\PSA$-equivalent to
\begin{equation} \label{elim1}
p(1,\bfy)\cdot  p(0,\bfy)\ =\ 0\ \wedge\  (\exists z) \big[x \ =\ 
z'\cdot p(0,\bfy)\ \plus \  z\cdot p(1,\bfy)' \big] .
\end{equation}
$(\exists x)\big[p(x,\bfy) = 0\big]$ is $\PSA$-equivalent to 
\begin{equation} \label{eliminant}
p(1,\bfy)\cdot  p(0,\bfy)\ =\ 0.
\end{equation}
\item
Thus \eqref{eliminant} is the complete result of eliminating $x$ 
from the equation $p(x,\bfy)=0$;
 if this condition holds, then 
\begin{equation}\label{gen sol}
x \ =\ z'\cdot  p(0,\bfy)\ \plus\   z\cdot  p(1,\bfy)' 
\end{equation}
gives the general solution
of $p(x,\bfy)=0$ for $x$.
\end{theorem}

\begin{proof}
The equation $p(x,\bfy)=0$ can be written as
$$
p(1,\bfy)\cdot  x \ \plus  \ p(0,\bfy)\cdot  x'\ =\ 0,
$$
which is equivalent to 
$$
p(1,\bfy)\cdot  x = 0\ \wedge\  p(0,\bfy)\cdot  x' =  0, 
$$
which can be written as
\begin{equation} \label{sol cond}
p(0,\bfy)  \subseteq   x  \subseteq p(1,\bfy)'.
\end{equation}
There is an $x$ which makes \eqref{sol cond}
 true iff $p(0,\bfy) \subseteq p(1,\bfy)'$, that is,
iff $p(1,\bfy)\cdot  p(0,\bfy) = 0$. 
If \eqref{sol cond} is fulfilled, then\\
 (a) 
$x = x'\cdot  p(0,\bfy) \plus   x\cdot  p(1,\bfy)' $, 
so there is a $z$ as required by \eqref{elim1};
and\\
 (b) if $x = z'\cdot  p(0,\bfy) \plus   z \cdot  p(1,\bfy)' $ for some $z$, then clearly
\eqref{sol cond} holds.
\end{proof}

Schr\"oder (p.~460 of \cite{Schr}, Vol. I) credits Boole with the previous elimination theorem, 
calling it Boole's Main Theorem, even though Boole claimed this result for his own algebra of
logic, not the modern one presented here. 
Likewise we credit  Boole with the next result.
\begin{corollary} [{\bf Boole's Elimination Theorem}]
$(\exists \bfx)\big[p(\bfx,\bfy) = 0\big]$ is $\PSA$-equivalent to 
\begin{equation} \label{GenElim}
0\ =\ \prodkt_{\sigma\in 2^{\wtk}} p(\sigma,\bfy),
\end{equation}
the complete result of eliminating $\bfx$ from $p(\bfx,\bfy)=0$.
\end{corollary}
Schr\"oder does not give a general formula to find the solution for $\bfx$ 
in $p(\bfx,\bfy)=0$ as a function of $\bfy$ when $\bfx$ is a list of more 
than one variable.

\subsection{Schr\"oder's Elimination Program}
Schr\"oder goes on to set up an ambitious program to tackle 
quantifier elimination for arbitrary open (i.e., quantifier-free) formulas
$\omega(\bfx,\bfy)$. Every open formula is equivalent to a disjunction of
conjunctions of basic formulas. Since
$$
(\exists \bfx) \big[ \varphi_1(\bfx,\bfy)\vee\ \ldots\ \vee \varphi_m(\bfx,\bfy) \big]
$$
is equivalent to
$$
(\exists \bfx) \big[ \varphi_1(\bfx,\bfy) \big]\vee \ \ldots\ \vee
(\exists \bfx) \big[ \varphi_m(\bfx,\bfy) \big] ,
$$
quantifier elimination for open formulas reduces to quantifier elimination 
for conjunctions of basic formulas, that is, to formulas of the form
$$
(\exists \bfx) \big[ p(\bfx,\bfy)=0\ \wedge\ 
q_1(\bfx,\bfy) \neq 0\ \wedge\ \cdots \ \wedge\ q_n(\bfx,\bfy) \neq 0 \big].
$$
(Only one equation is needed in view of the Reduction Theorem.)

\section{Elimination and Solution for Basic Formulas}
Schr\"oder extended his Theorem \ref{E and S} to the case of one equation
and one negated equation (see Corollary \ref{one-one} below). 
Our next result extends the parametric solution portion of Schr\"oder's result
to include any number of negated equations.

\subsection{Parametric solutions to systems of basic formulas}
\begin{theorem} \label{parametric solns}
Given a system 
\begin{equation} \label{sys}
p(x,\bfy)=0, \ q_1(x,\bfy)\neq 0,\ \ldots,\ q_n(x,\bfy)\neq 0
\end{equation}
of basic formulas, write them in the form  {\rm (by setting $a(y)= p(1,y), b(y)= p(0,y)$, etc.)}
\begin{align*}
a(\bfy) \cdot   x   \  \plus\ b(\bfy) \cdot  x'         &= 0\\
c_1(\bfy) \cdot   x\ \plus\ d_1(\bfy) \cdot  x' &\neq 0\\
\ \vdots&\\
c_n(\bfy) \cdot  x\ \plus\ d_n(\bfy) \cdot  x'   &\neq 0.
\end{align*}
Let $\varphi(\bfy,\bfv)$ be the conjuction of the formulas:
\begin{align*}
a(\bfy) \cdot  b(\bfy)  &= 0 \nonumber\\
0 \neq v_i &\subseteq c_i(\bfy) \cdot   a'(\bfy) \ \plus\ d_i(\bfy) \cdot   
b'(\bfy)  \nonumber\\
 v_i\cap v_j &\subseteq c_i(\bfy)\ \plus\ d_j(\bfy) \cdot   c_j(\bfy) \ 
 \plus\ d_i(\bfy) \quad \text{for }i \neq j
\end{align*}
Then the system \eqref{sys} is $\PSA$-equivalent to
\begin{align}\label{eform}
(\exists w)&(\exists v_1)\cdots (\exists v_k)\bigg[
\ \varphi \ \wedge \ 
\displaystyle  \bigg(x = w\cdot \Big(a \plus \summe_{i\in \wtk} v_i  {c_i}'\Big)' \ \plus\ 
w\,' \cdot \Big(b \plus \summe_{i\in \wtk} v_i {d_i}'\Big)\bigg)
\bigg].
\end{align}
\end{theorem}

\begin{proof}
The original system \eqref{sys} is clearly $\PSA$-equivalent to
$$
(\exists v_1) \cdots (\exists v_k) \bigg[
 \bigwedge_{j\in \wtk} (v_j \neq 0)\ \wedge \ 
 (p = 0)\ \wedge\ \bigwedge_{j\in\wtk} (v_j \subseteq q_j) \bigg] .
$$
The formula
$$
 (p = 0)\ \wedge\ \bigwedge_{j\in\wtk} \big(v_j \subseteq q_j \big)
$$
is $\PSA$-equivalent to the conjunction of equations
$$
 (p = 0)\ \wedge\ \bigwedge_{j\in\wtk} \big(v_j  q_j' = 0 \big).
$$
Reduce this, by Theorem \ref{reduction}, to a single equation
$$
 \big( a  x \cup b x' \big)\ \plus \ 
 \Big( \summe_{j\in\wtk}v_j \cdot \big(c_j x \plus d_j x'\big)'  \Big) \ =\ 0,
$$
and expand it about $x$ to obtain
$$
\Big( a \plus \summe_{j\in\wtk} v_j  {c_j}'\Big)\cdot  x \ \plus\ 
\Big( b \plus \summe_{j\in\wtk} v_j  {d_j}'\Big)\cdot  x' \ =\ 0.
$$
This is $\PSA$-equivalent, by Theorem \ref{E and S}, to the conjunction of the
two formulas
\begin{eqnarray}
0 &=& \Big( a \plus \summe_{j\in\wtk} v_j  {c_j}'\Big)  \cdot 
\Big( b \plus \summe_{j\in\wtk} v_j  {d_j}'\Big) \label{eqpt}\\
(\exists w) \bigg[ x &=&  w \cdot   \Big( a \plus \summe_{j\in\wtk} v_j  {c_j}'\Big)'
\ \plus\  w\,'  \cdot  \Big( b \plus \summe_{j\in\wtk} v_j  {d_j}'\Big)\bigg].
\end{eqnarray}
Expanding \eqref{eqpt} as a polynomial in the $v_j$ transforms 
it into the conjunction of the equations
\begin{eqnarray*}
0&=&a  b\\
0 &=& v_j  \cdot  \big( a  {d_j}' \plus b  {c_j}' 
\plus {c_j}' {d_j}' \big) \quad \text{for }j \in \wtk\\
0 &=& v_i  v_j  \cdot  \big( {c_i}'  {d_j}' \plus {c_j}'  {d_i}'  \big)
 \quad \text{for }i, j \in \wtk, i\neq j,
\end{eqnarray*}
which can be rewritten as
\begin{eqnarray*}
0&=& a  b\\
0 &=& v_j  \cdot  \big( c_j  a' \plus d_j  b\,' \big)' \qquad \qquad \text{for }j \in \wtk\\
0 &=& v_i  v_j \cdot  \big( (c_i \plus d_j)  \cdot  (c_j \plus d_i) \big)'
 \qquad \text{for }i, j \in \wtk, i\neq j.
\end{eqnarray*}
Thus \eqref{sys} is $\PSA$-equivalent to the existence of $\bfv$ and $w$ such that:   
\begin{eqnarray*}
0&=& a  b\\
0 &\neq & v_j\ \subseteq \  c_j  a' \plus d_j  b\,' \quad \text{for }j \in \wtk\\
v_i  v_j &\subseteq&  (c_i \plus d_j) \cdot   (c_j \plus d_i) 
\quad \quad \quad \text{for } i, j \in \wtk, i\neq j\\
x &=& \Big[ w \cdot   \Big( a \plus \summe_{j\in\wtk} v_j  {c_j}'\Big)'\,\Big]
\ \plus\ \Big[ w\,'  \cdot  \Big( b \plus \summe_{j\in\wtk} v_j  {d_j}'\Big)\Big].
\end{eqnarray*}

\end{proof}
Note that in \eqref{eform}, the restrictions on the parameters
 $v_i$ are very simple, and there is no restriction on the parameter $w$. 

If one restricts the above to the case where there is exactly one
negated equation, then one has a full elimination result, extending
Theorem \ref{E and S}. (See p.~205-209 of \cite{Schr} Vol. II)
\begin{corollary} [{\bf Schr\"oder}] \label{one-one}
The formula
\begin{equation}\label{PN1}
(\exists x) \big[p(x,\bfy) =0\ \wedge\ q(x,\bfy) \neq 0\big]
\end{equation}
 is $\PSA$-equivalent to 
\begin{equation}\label{PN2}
\big[p(1,\bfy)\cdot p(0,\bfy)\ =\ 0 \big]\ \wedge\ 
\big[ q(1,\bfy) \cdot p(1,\bfy)' \ \plus   \ q(0,\bfy) \cdot p(0,\bfy)'\ \neq\ 0 \big].
\end{equation}
\end{corollary}
\begin{proof}
By Theorem \ref{parametric solns},
\begin{equation}\label{PN1}
(\exists x) \big[p(x,\bfy) =0\ \wedge\ q(x,\bfy) \neq 0\big]
\end{equation}
is equivalent to the following, where mention of the $\bfy$'s 
has been suppressed:
$$
(\exists z) 
\Big[ \big(
p(1) \cdot  p(0)  = 0 \big) \ \wedge\ 
\big(0 \neq z \subseteq q(1) \cdot p(1)' \ \plus\ q(0) \cdot b(0)'  \big) \Big],
$$
which in turn is equivalent to \eqref{PN2}.
\end{proof}

The difficulties with quantifier elimination for a conjunction of basic formulas
starts with two negated equations. The simplest example to illustrate this is
try to eliminate $x$ from 
\begin{equation} \label{2NE}
(\exists x)\Big(x y \neq 0 \ \wedge x' y \neq 0\Big).
\end{equation}
Clearly $y \neq 0$ follows from \eqref{2NE}.
However this is not equivalent to \eqref{2NE} because there will 
be an $x$ as in \eqref{2NE} iff
$y$ has at least 2 elements. The formula \eqref{2NE} is equivalent to
the expression
\begin{equation} \label{2QE}
 |y| \ge 2.
\end{equation}
Unfortunately there is no way to express $|y|\ge 2$ by an open formula
 in the language of power-set algebra that we are using 
 (with fundamental operations $\plus  ,\cap,'$ and constants $0,1$). 
 In modern terminology, power-set algebra does not admit elimination of quantifiers.

Nonetheless, Schr\"oder struggled on, showing how admitting symbols for 
{\em elements}
of sets would allow him to carry out elimination for the case of eliminating one
variable. Then he says that the result for eliminating two variables would follow
similar reasoning, but would be much more complicated, etc. 
(See \S49 in \cite{Schr} Vol. II.)

In 1919 Skolem \cite{Skolem-1919} gave an elegant improvement on Schr\"oder's 
approach to elimination by showing that if one adds the predicates 
$|\ |\ge n$, for $n\ge 0$, then this augmented version of power-set 
algebra has a straightforward procedure for the elimination of quantifiers.

\subsection{Using Boole's $V$ in Elimination} \label{sec on Velim}
In $\S$\ref{sec on V} we saw that one could replace negated equations by $V$-equations,
in the spirit of Boole, when studying the  validity of arguments. 
Now we ask if one can do the same when working with elimination. 
Since Schr\"oder's 
elimination, in the original language of power-set algebras, halts with one negated
equation, we will restrict our attention to considering a $V$-version of 
Theorem \ref{one-one}. 
We will use a $V$-translation to convert the negated equation in the premiss into an equation,
apply Boole' elimination result to the two premisses, split this into an equation and a 
$V$-equation, and then convert the $V$-equation back into a negated equation. 
It will be noted
that this method gives the correct answer found by Schr\"oder.
Thus,
for example, one can apply the $V$-method to derive the valid syllogisms which
have one particular premiss and one universal premiss.

\begin{theorem} \label{V one-one}
The $V$-translation of 
\begin{equation}\label{PN1 V0}
(\exists x) \big[p(x,\bfy) =0\ \wedge\ q(x,\bfy) \neq 0\big]
\end{equation}
is
\begin{equation}\label{PN1 V1}
(\exists x) \big[p(x,\bfy) =0\ \wedge\ V \cdot q'(x,\bfy) = 0\big].
\end{equation}
Eliminating $x$ using Theorem \ref{E and S} gives
\begin{equation}\label{PN1 V2}
p(1,\bfy) \cdot p(0,\bfy)\ \plus  \ V \cdot \Big(q(1,\bfy) \cdot p(1,\bfy)' \ \plus  
\ q(0,\bfy) \cdot p(0,\bfy)' \Big)'\ =\ 0,
\end{equation}
which translates back into
\begin{equation}\label{PN1 V3}
\big[p(1,\bfy) \cdot p(0,\bfy)\ =\ 0 \big]\ \wedge\ 
\big[ q(1,\bfy) \cdot  p(1,\bfy)' \ \plus   \ q(0,\bfy) \cdot  p(0,\bfy)'\ \neq\ 0 \big].
\end{equation}
\end{theorem}

\begin{proof}
The $V$-translation converts 
$$
p(x,\bfy) =0\ \wedge\ q(x,\bfy) \neq 0
$$
into
$$
p(x,\bfy) =0\ \wedge\ V \cdot q(x,\bfy)' = 0.
$$
Reducing this to a single equation gives
$$
p(x,\bfy)\ \plus\   V \cdot q(x,\bfy)' = 0.
$$
The complete result of eliminating $x$ is, by Theorem \ref{E and S},
$$
\Big( p(1,\bfy)\ \plus\   V \cdot q(1,\bfy)' \Big)  \cdot  \Big( p(0,\bfy)\ \plus \  
V \cdot q(0,\bfy)' \Big) \ =\ 0.
$$
Multiplying this out gives
$$
p(1,\bfy)  \cdot p(0,\bfy)\ \plus  \ V \cdot  \Big(p(1,\bfy) \cdot q(0,\bfy)' \ \plus  \ 
p(0,\bfy) \cdot q(1,\bfy)' \ \plus  \ q(0,\bfy)' \cdot q(1,\bfy)'\Big)\ =\ 0,
$$
which is equivalent to the two equations
\begin{align}
p(1,\bfy)  \cdot p(0,\bfy) &= 0  \label{Veq1}\\
V \cdot  \Big(p(1,\bfy) \cdot q(0,\bfy)' \ \plus  \ p(0,\bfy) \cdot q(1,\bfy)' \ \plus  
\ q(0,\bfy)' \cdot q(1,\bfy)'\Big)\ &=\ 0. \label{Veq2}
\end{align}
In view of \eqref{Veq1}, 
the equation \eqref{Veq2} is equivalent to
\begin{align*}
V \cdot \Big(q(1,\bfy) \cdot p(1,\bfy)' \ \plus  \ q(0,\bfy) \cdot p(0,\bfy)' \Big)'\ &=\ 0,\\
\intertext{which translates back to the negated equation}
q(1,\bfy) \cdot p(1,\bfy)' \ \plus  \ q(0,\bfy) \cdot p(0,\bfy)' \ &\neq\ 0.
\end{align*}
This, combined with \eqref{Veq1}, gives \eqref{PN1 V3}, 
the same result as in Theorem \ref{one-one}.
\end{proof}

\begin{eqnarray*}
  t(\bfx) 
\ = \  \begin{cases}
0&\text{if }\quad \PSA \models t(\bfx)=0\\
\summe_{\substack{\sigma \in 2^{\wtk}\\t(\sigma)\neq 0}}  \cC_\sigma(\bfx) 
& \text{otherwise}.
\end{cases}
\end{eqnarray*}

\end{document}